\documentclass[11pt]{article}
\usepackage{amsmath}
\usepackage{float}
\usepackage{amsfonts}
\usepackage{amsthm}
\usepackage{amssymb, setspace}
\usepackage{color,graphicx,epsfig,geometry,hyperref,fancyhdr}
\usepackage[T1]{fontenc}
\usepackage{subfig}
\usepackage{comment}
\usepackage{commath}
\usepackage{bbm}
\usepackage{listings}
\usepackage{mathrsfs,fleqn}
\newtheorem{theorem}{Theorem}[section]

\newtheorem{lemma}{Lemma}[section]

\newcommand{\N}{\mathbb{N}}
\newcommand{\Z}{\mathbb{Z}}
\newcommand{\R}{\mathbb{R}}
\newcommand{\C}{\mathbb{C}}
\newcommand{\pdr}{\partial_r}

\newcommand{\e}{\text{ev}}
\newcommand{\p}{\text{pr}}
\newcommand{\paren}[1]{\left( #1 \right) }
\DeclareSymbolFont{matha}{OML}{txmi}{m}{it}
\DeclareMathSymbol{\varv}{\mathord}{matha}{118}

\graphicspath{{paper-pics/}}

\begin{document}

\begin{flushleft}
\Large 
\noindent{\bf \Large Well-posedness for the biharmonic scattering problem for a penetrable obstacle}
\end{flushleft}

\vspace{0.2in}
{\bf  \large Rafael Ceja Ayala}\\
\indent {\small School of Mathematical and Statistical Sciences, Arizona State University, Tempe, AZ 85287. }\\
\indent {\small Email:  \texttt{rcejaaya@asu.edu} }\\

{\bf  \large Isaac Harris}\\
\indent {\small Department of Mathematics, Purdue University, West Lafayette, IN 47907. }\\
\indent {\small Email: \texttt{harri814@purdue.edu} }\\

{\bf  \large Tonatiuh S\'anchez-Vizuet}\\
\indent {\small Department of Mathematics, The University of Arizona, Tucson, AZ  85721. }\\
\indent {\small Email: \texttt{tonatiuh@arizona.edu} }\\


\begin{abstract}
\noindent We address the direct scattering problem for a penetrable obstacle in an infinite elastic two-dimensional Kirchhoff--Love plate. Under the assumption that the plate's thickness is small relative to the wavelength of the incident wave, the propagation of perturbations on the plate is governed by the two-dimensional biharmonic wave equation, which we study in the frequency domain. With the help of an operator factorization, the scattering problem is analyzed from the perspective of a coupled boundary value problem involving the Helmholtz and modified Helmholtz equations. Well-posedness and reciprocity relations for the problem are established. Numerical examples for some special cases are provided to validate the theoretical findings.\\
\end{abstract}

\noindent \textbf{Keywords:} Biharmonic scattering, reciprocity relations, Helmholtz equation, modified Helmholtz equation, Kirchhoff--Love plate, bi-Laplacian operator.\\

\noindent \textbf{AMS subject classifications(2020):} 47A40, 74J05, 74H25, 35J35.

\section{Introduction}

In this paper, we study the direct biharmonic scattering problem for a penetrable obstacle in an infinite two-dimensional elastic Kirchhoff-Love plate. In recent years, the study of biharmonic problems has garnered significant attention due to their wide applications in nondestructive testing. For instance, ultra-broadband elastic cloaking devices \cite{nonlinearcloak,elasticcloak} which has applications ranging from the car industry to anti-earthquake passive systems for smart buildings (see \cite{elasticcloak} for more details). Some studies have touched on the global uniqueness problem for the recovery of potential functions or medium parameters associated with the biharmonic operators  \cite{firstorderper,orderperbiharmonic,partialdatainverse}. These studies have promising practical applications such as brain imaging \cite{IBPBI,scatteringformulti} and problems associated with geodynamical systems which arise in geomagnetic anomaly detection \cite{maganomalies,maganomalies2}. Also of interest are the reconstruction of densities and internal sources by applying active measurements \cite{passivemeasure} (where one actively sends a priori known probing sources and collects the responses for identification purposes) as well as the analysis of inverse problems for biharmonic scattering problems \cite{inversebackscattering}. Although some theory and results have been developed \cite{BIE-biharmonic,logestimates}, work remains to be done for the case of biharmonic wave scattering and transmission.

Motivated by the work \cite{biwellposed,novelbi} and their analysis on the biharmonic wave clamped scattering problem, we address the problem concerning continuous traces at the boundary of the medium. One of the main differences between these references and the present manuscript is that we deal with the transmission problem that arises when dealing with a penetrable obstacle, while they consider the impenetrable case only. Although the well-posedness of the constant-coefficients case of this problem can be considered a particular case of the Agmon-H\"ormander framework for differential equations with simple charactristics \cite{AgHo}, our result generalizes also to equations with variable coefficients. Moreover, even for the constant-coefficients case, here we present an explicit direct argument that, to the best of our knowledge, was absent from the literature until this effort. We analyze the problem through a splitting of the biharmonic wave operator \cite{highlyacc,frameworkmulti}, where two auxiliary functions play a crucial part on the development of the theory; these are the Helmholtz and modified Helmholtz wave functions \cite{optiforinversebiha}. We address the uniqueness of the scattering problem using Rellich's lemma and the exponentially decaying property of the modified Helmholtz wave function. The analysis includes showing that the problem is of Fredholm type and establishing the well-posedness of the scattering problem. In many mathematical investigations, changing the form of incident field can dramatically alter the analysis and the results of the scattering problem. In this study we consider the cases when the incident field is either a plane wave or a point source. We also study two important reciprocity relationships for the scattered field: one of which deals with far-field data and the other considers near-field data. Relationships of this kind are well known for acoustic scattering but have not been studied for the biharmonic scattering problem.

The paper is organized as follows: the next section outlines the direct scattering problem under consideration, detailing the conditions of the medium. Section \ref{uniq} delves into the uniqueness of the scattering problem under consideration. In Section \ref{Fred}, we show that the problem is Fredholm of index 0. We address two classical reciprocity relations for point sources and plane waves in Section \ref{recip-rel}. Finally, the last section presents numerical experiments to verify and validate the theoretical findings for penetrable scatterers.

\section{Formulation of the problem}\label{direct-prob-LDSM}
Consider that a scatterer occupying the domain $D\subset \mathbb R^2$ with $C^2$-smooth boundary $\partial D$ is compactly contained in the interior of a region enclosed by a curve $\Gamma$ such that $\overline{D} \subset \text{int}(\Gamma)$. We will refer to $\Gamma$ as the measurement curve.  The scatterer is illuminated by a known incident field, $u^i$, which will be assumed to be either
$$
u^i(x, d)=\text{e}^{\text{i}kx \cdot d} \quad \text{a plane wave with the incident direction $d \in \mathbb{S}^{1}$}
$$ 
or 
$$
u^i(x, {z})= \mathbb{G}(x,z;k) \quad \text{ a point source located at $z \in  \Gamma$}.
$$
The function 
\begin{equation}\label{fundamentalsol}
\mathbb{G}(x,z;k) =\frac{\text{i}}{8k^2} \left[ H^{(1)}_0(k|x-z|) -  H^{(1)}_0(\text{i}k|x-z|)\right] \quad \text{for $x\neq z$}
\end{equation}
is the `radiating' fundamental solution to the operator $(\Delta^2 -k^4)$, and $H^{(1)}_\ell$ denotes the Hankel function of first kind of order $\ell$ and $k>0$ is the wavenumber. Both incident fields depend on one parameter: the incident direction $d \in \mathbb{S}^1$ in the case of the plane wave, and the location of the source $z \in \Gamma$  for the point source. This dependence has been made explicit in the notation.

The interaction of the incident field and the scatterer denoted by $D$ produces a total wavefield $u$ consisting of the superposition of the incident field $u^i$ and a radiating scattered field $u^s(x) \in H^2_{loc}(\mathbb{R}^2)$ that satisfies 
\begin{align}
\Delta^2 u^s-k^4 n(x)u^s=-k^4 \big(1-n(x)\big) u^i \quad & \text{in} \hspace{.2cm}  \mathbb{R}^2  \label{direct2}.
\end{align}
The total field $u$ satisfies the transmission conditions 
\begin{align}
 [\![ u ]\!] =[\![\partial_\nu u ]\!] =[\![\Delta u ]\!] =[\![\partial_\nu \Delta u ]\!] =0\,, \label{direct3}
\end{align}
where, for a given function $\phi$, the notation 
$$
[\![\phi ]\!] := \phi_{+}\big|_{\partial D} - \phi_{-}\big|_{\partial D}
$$
denotes the discontinuity across the boundary $\partial D$, and the operators $_+\big|_{\partial D}$ and $_-\big|_{\partial D}$ represent the trace, or restriction, taken  respectively from $\mathbb R^2\setminus\overline{D}$ and $D$. In addition,the scattered field and its Laplacian are required to satisfy the Sommerfeld radiation conditions
\begin{equation} \label{direct4}
\lim_{|x| \to \infty} \sqrt{r}(\pdr{u^s} - \text{i} k u^s) = 0,  \quad \text{ and } \quad \lim_{|x| \to \infty} \sqrt{r}(\pdr{\Delta u^s} - \text{i} k \Delta u^s) = 0
\end{equation}
for $r :=|x|$ which is satisfied uniformly for all directions $\hat{x}:=x/|x|\in \mathbb{S}^1$. We will require that the coefficient $n(x)$ is such that supp$(n(x)-1)=D$ with $n \in L^{\infty}(D)$ and that
\begin{equation}\label{absorbing}
\text{Im}(n(x))\geq C>0 \quad \text{ for a.e. $x \in  D$}.
\end{equation}
This condition is known in the literature as the ``absorbing obstacle condition'' i.e. similar to acoustic scattering problems.

The main goal of this paper is to prove the well-posedness for the direct scattering problem \eqref{direct2}--\eqref{direct4} and to derive reciprocity relations for the quantities 
$$
u^{\infty}(\hat{x},d) \,\, \textrm{ for $\hat{x},d \in \mathbb{S}^1$ } \quad \text{ and } \quad u^s(x,z)\,\, \textrm{ for $x,z \in \Gamma$}.
$$
These types of relations are often used in studying the inverse problem of recovering the scatterer from either far-field or near-field data (see e.g. \cite{DongLi23}).
 
Just as in \cite{biwellposed,optiforinversebiha,novelbi} we will make use of the fact that the scattered field $u^s$ in $ \mathbb{R}^2\setminus\overline{D}$ can be expressed as the sum of a propagating field $u_\p$ and and evanescent field $u_\e$ which, respectively, satisfy the Helmholtz equation (wave number $k$) and modified Helmholtz equation (wave number $\text{i}k$)
\begin{align}
     (\Delta +k^2)u_\p=0\quad \text{ and } \quad (\Delta-k^2)u_\e=0\quad\text{in}\hspace{.2cm} \mathbb{R}^2\setminus\overline{D}\,,
     \label{hemholtzmodified}
\end{align}
along with the Sommerfeld radiation condition for $u_\p$, i.e.,
\begin{align}
\lim_{r \to \infty} \sqrt{r}(\pdr{u_\p} - \text{i} k u_\p) = 0. 
\label{radiupue}
\end{align}
Rewriting the biharmonic operator using the factorization $\Delta - k^4 = (\Delta-k^2)(\Delta+k^2)$ it follows that if $u^s=u_{\text{pr}} + u_{\text{ev}}$, then 
$$
\left(\Delta^2 + k^4\right)u^s = (\Delta-k^2)(\Delta+k^2)u_{\text{pr}} + (\Delta+k^2)(\Delta-k^2)u_{\text{ev}}=0 \quad \text{in} \,\, \mathbb{R}^2\setminus\overline{D},
$$
proving that $u^s = u_{\text{pr}} + u_{\text{ev}}$ indeed satisfies the biharmonic problem. 

To show that the decomposition indeed satisfies the Sommerfeld radiation condition we will start with the evanescent component and proceed by a Fourier expansion. Let $B_R$ be a disk of radius $R>0$ centered at the origin and such that $\overline{D} \subset B_R$. In the exterior of $B_R$, separation of variables can be used to show that
$$
u_{\text{pr}}(r,\theta)=\sum_{|n|=0}^{\infty}\frac{H^{(1)}_\ell (kr)}{H^{(1)}_\ell (kR)}u^{(\ell)}_{\text{pr}}\text{e}^{\text{i}\ell \theta} 
\qquad \text{ and } \qquad
u_\e( r,\theta)=\sum_{|n|=0}^{\infty}\frac{H^{(1)}_\ell (\text{i}kr)}{H^{(1)}_\ell (\text{i}kR)}u^{(\ell)}_{\text{ev}}\text{e}^{\text{i}\ell\theta},
$$
where $u_{\text{pr}}^{(\ell)}$ and $u^{(\ell)}_{\text{ev}}$ denote the $\ell$-th Fourier coefficients of $u_{\text{pr}}(R,\theta)$ and $u_\e(R,\theta)$, respectively. Using the well-known facts (see e.g. \cite{NIST}) for all $\ell\in \mathbb{Z}$
\begin{enumerate}
\item[(i)] ${\displaystyle | H^{(1)}_\ell (\text{i} kr)| = \mathcal{O}\big(r^{-1/2} \text{e}^{-kr}\big) \quad \text{as} \,\, r \to \infty\,,}$

\item[(ii)] ${\displaystyle \frac{\mathrm{d}}{\mathrm{d}t}H^{(1)}_\ell(t)=\frac{1}{2} \Big(H^{(1)}_{\ell-1}(t)+H^{(1)}_{\ell+1}(t)\Big)}$\,,
\end{enumerate}
and the Fourier expansion for $u_{\text{ev}}$, we can show that $\partial_r u_\e$ decays exponentially as $r\to \infty$. This, together with the observation that $u_\p$ is a radiating solution to the Helmholtz equation, implies that the scattered field $u^s$ has the following asymptotic behavior
\[
u^s(x)=\frac{\text{e}^{\text{i}\pi/4}}{\sqrt{8\pi k}} \frac{\text{e}^{\text{i}k|x|}}{\sqrt{|x|}}\left\{u^{\infty}(\hat{x})+\mathcal{O}\left(\frac{1}{|x|}\right) \right\}\hspace{.3cm}\text{as}\hspace{.3cm}|x|\longrightarrow \infty\,,
\]
where $u^{\infty}(\hat{x})$ denotes the far-field pattern of the scattered field $u^s$. The asymptotic expression above implies that $u^{\infty}(\hat{x})=u_\p^{\infty}(\hat{x})$ and therefore (see e.g. \cite{HLL-biharmonic})
\begin{align}
    u^{\infty}(\hat{x})=\int_{\partial B_R}\Big(u_\p(y)\partial_{\nu_y}\text{e}^{-\text{i}k\hat{x}\cdot y}-\partial_{\nu_y}u_\p(y) \text{e}^{-\text{i}k\hat{x}\cdot y}\Big) \, \text{d}s(y)\label{far-fieldpattern}.
\end{align}
Notice the far-field data does not contain any information about the function $u_\e$ as it decays exponentially. 

\color{black}
\section{Uniqueness of the biharmonic problem}\label{uniq}
In this section, we will address the question of uniqueness for the direct scattering problem  \eqref{direct2}--\eqref{direct4}. To this end, we will consider the case when the incident field $u^i=0$, and ask if a solution to \eqref{direct2}--\eqref{direct4} vanishes for a.e. $x \in \mathbb{R}^2$. Just as in \cite{biwellposed,novelbi}, we will first prove uniqueness and then we will show that the problem is Fredholm of index zero to obtain the existence of a solution. 

Recall that $B_R$ is the ball of radius $R>0$ centered at the origin and define the Hilbert space 
\begin{equation}\label{hilbertspace}
H^2\left(B_R,\Delta^2\right) := \big\{u\in H^2(B_R) \, \,:\, \Delta^2 u\in L^2(B_R) \big\}
\end{equation}
equipped with the standard graph norm and inner-product. To study the variational formulation of the scattering problem in this function space, we will use the surface differential operators on $\partial B_R$
\begin{equation}\label{defN}
\begin{aligned}
Mu :=\,&\mu \Delta u-(1-\mu)\left(\frac{1}{R}{\partial_r u}+\frac{1}{R^2}{\partial^2_{\theta} u} \right) \label{defM}\,,\\[1.5ex]
Nu :=\,&-{\partial_r}\Delta u-(1-\mu )\frac{1}{R^2}{\partial_\theta}\left({\partial_{\theta}\partial_r u}-\frac{1}{R}{\partial_{\theta} u}\right)\,,
\end{aligned}
\end{equation}
which are the spherical cases of those defined in \cite[Equation 1.3]{biwellposed}. Above, the parameter $\mu$ (Poisson's ratio) will be considered to be a constant in the interval $[0,1/2)$, and $\partial_r$ and $\partial_\theta$ are the radial and angular derivatives in polar coordinates. The mapping properties of these operators in Sobolev spaces have been extensively studied in \cite{HsWe2021}. In particular, \cite[Lemma 5.7.1]{HsWe2021} guarantees that
\[
u \mapsto \big(Nu , Mu\big)^{\top}
\]
is a continuous linear mapping from $H^2(B_R,\Delta^2)$ into $H^{-3/2}(\partial B_R)\times H^{-1/2}(\partial B_R)$, while \cite[Theorem 5.7.3]{HsWe2021} ensures that the mapping has a bounded right inverse. With this, just as in \cite{biwellposed}, by appealing to Green's identities and defining for a fixed value of $\mu$
\begin{equation}
a(u,\varphi):=\int_{B_R}\Big( \mu \Delta u\Delta \overline{\varphi}+(1-\mu ) \sum_{i,j=1}^2{\partial_{x_{i} x_{j}} u} \, {\partial_{x_{i} x_{j}} \overline{\varphi}}\Big) \, \text{d}x\,, \label{def-varA}
\end{equation}
we have that 
$$
 \int_{B_R} \overline{\varphi} \Delta^2 u\, \text{d}s=a(u,\varphi)-\int_{\partial B_R} \Big(\overline{\varphi}Nu^s+\partial_r \overline{\varphi} Mu^s\Big) \, \text{d}s
$$
for all $u \in H^2(B_R,\Delta^2)$ and $\varphi \in H^2(B_R)$. The first term inside the integral over $\partial B_R$ can be interpreted as a duality pairing between $H^{3/2}(\partial B_R)$ and $H^{-3/2}(\partial B_R),$ while the second term can be seen as a duality pairing between $H^{1/2}(\partial B_R)$ and $H^{-1/2}(\partial B_R).$

Now, recall the direct problem given by \eqref{direct2}--\eqref{direct4} for the scattered field $u^s$. From theconditions biharmonic equation in \eqref{direct2} we know that the scattered field $u^s \in H^2(B_R,\Delta^2)$ satisfies 
\begin{align}
a \big(u^s,\varphi \big)-k^4\int_{B_R}n u^s\overline{\varphi} \, \text{d}x - \int_{\partial B_R} \Big(\overline{\varphi}Nu^s+\partial_r \overline{\varphi} Mu^s\Big) \, \text{d}s = k^4\int_{D} (n-1)u^i \overline{\varphi}\, \text{d}x \label{varform1}
\end{align}
for all $\varphi \in H^2(B_R)$. On the volume integral on the right hand side, we have used the assumption that supp$(n-1)=D$. Notice, that \eqref{varform1} is a variational formulation of the direct scattering problem in $B_R$ that does not incorporate the radiation condition. {By using this variational formulation along with the ``absorbing obstacle'' condition \eqref{absorbing} on the coefficient $n$, } we can now show that the when the incident field is trivial the scattered field is necessarily trivial.

\begin{theorem} \label{up=0}
The direct scattering problem \eqref{direct2}--\eqref{direct4} has at most one solution.
\end{theorem}
\begin{proof}
To prove the claim, since the scattering problem \eqref{direct2}--\eqref{direct4} is linear we only need to show that $u^i=0$ implies that $u^s=0$. To this end, we take $\varphi=u^s$ and $u^i=0$ in \eqref{varform1} to obtain that 
$$
a(u^s,u^s)-k^4\int_{B_R}n|u^s|^2 \, \text{d}x-\int_{\partial B_R}\Big(\overline{u^s}Nu^s+{\partial_r \overline{u^s}}Mu^s \Big)\,  \text{d}s=0
$$
for any fixed constant parameter $\mu$. Taking the imaginary part of the expression above yields 
$$
\text{Im}\Big[ \int_{\partial B_R}\Big(\overline{u^s} Nu^s+{\partial_r \overline{u^s}}Mu^s\Big) \, \text{d}s\Big]=\text{Im}\Big[ a \big(u^s,u^s)-k^4\int_{B_R}n|u^s|^2\text{d}x \Big]
$$
from which, appealing to our assumption that $\text{Im}[n(x)]\geq C >0$ in the scatterer $D$ and $n=1$ outside the scatterer $\mathbb{R}^2\setminus\overline{D}$, we have
$$
\text{Im}\left[ a\big(u^s,u^s\big)-k^4\int_{B_R}n|u^s|^2\text{d}x \right]= -k^4 \text{Im}\left[ \int_{D}n|u^s|^2\text{d}x \right]\leq 0.
$$
With this fact we can prove the uniqueness result in stages. First we will show that the propagating part of the scattered field, $u_\p$, is zero outside the scatterer.

The decomposition $u^s=u_\e+u_\p$ in $\mathbb{R}^2\setminus\overline{D}$ together with  
$$
(\Delta +k^2)u_\p=0 \quad \text{ and } \quad (\Delta-k^2)u_\e=0\quad\text{in}\hspace{.2cm} \mathbb{R}^2\setminus\overline{D}
$$
imply that that for any $\mu$ 
$$
M u^s = -k^2 u_\p +\mathcal{O} \big(R^{-3/2}\big) \quad \text{ and } \quad N u^s = k^2 \partial_r u_\p +\mathcal{O}\big(R^{-5/2}\big) \quad \text{ on $\partial B_R$ }
$$ 
as $R \to \infty$ (see Section 3 of \cite{biwellposed}). Therefore, the imaginary part of the boundary integral is given by 
\begin{align*}
    & \text{Im}\left[ \int_{\partial B_R}\Big(\overline{u^s} Nu^s+{\partial_r \overline{u^s}} Mu^s \Big)\, \text{d}s \right] \\
    &= -k^2 \text{Im}\left[ \int_{\partial B_R}\Big(u_\p {\partial_r \overline{u_\p}}-\overline{u_\p} {\partial_r u_\p} \Big) \, \text{d}s \right]+\mathcal{O}\big(R^{-1}\big)\\
    &=2k^2\text{Im}\left[\int_{\partial B_R} \overline{u_\p} {\partial_r u_\p} \, \text{d}s  \right]+\mathcal{O}\big(R^{-1}\big)\\
    &= 2k^3\int_{\partial B_R}|u_\p|^2 \, \text{d}s+2k^2\text{Im}\left[\int_{\partial B_R}\Big( {\partial_r u_\p} - \text{i} ku_\p \Big)\overline{u_\p} \, \text{d}s\right]+\mathcal{O}\big(R^{-1}\big).
\end{align*} 
Here, we have used the fact that both $u_\e$ and $\partial_r u_\e$ decay exponentially fast as $R \to \infty$ and the asymptotic behavior of $u_\p$ given by the radiation condition to obtain
$$
\int_{\partial B_R}|u_\p|^2 \, \text{d}s=\mathcal{O}(1)\quad \text{and} \quad \lim_{R \to \infty}\int_{\partial B_R}\left|{\partial_r u_\p}-\text{i} ku_\p\right|^2 \, \text{d}s=0.
$$  
From this, it follows that 
$$
\lim_{R \to \infty} 2k^3\int_{\partial B_R}|u_\p|^2\, \text{d}s = -k^4 \text{Im}\left[ \int_{D}n|u^s|^2\text{d}x \right] < 0.
$$
Therefore, by Rellich's lemma \cite[Theorem 3.5]{approach} the propagating part of the scattered field $u_\p $ must vanish in  $\mathbb{R}^2\setminus\overline{D}$. Notice that, by the above limiting equality, we also obtain that $u^s = 0$ in $D$ since the $u_\p = 0$ in  $\mathbb{R}^2\setminus\overline{D}$ and $\text{Im}[n(x)]\geq C >0$ in $D$. 

Now, we only need to show that the evanescent part of the scattered field $u_\e$ vanishes in $\mathbb{R}^2\setminus\overline{D}$. Using Green's first identity and the fact that $u_\e$ satisfies the modified Helmholtz equation we have
$$
-\int_{B_R\setminus \overline{D}}|\nabla u_\e|^2+k^s|u_\e|^2 \, \text{d}x+\int_{\partial B_R}{{u_\e} \partial_r \overline{u_\e} } \, \text{d}s+\int_{\partial D}{u_\e}{\partial_r\overline{u_\e}}\, \text{d}s=0.
$$ 
The transmission conditions in \eqref{direct3} forcing $u_\e$=0 on $\partial D$ together with the exponential decay of $u_\e$ and $\partial_r u_\e$ imply
$$
\int_{\R^2\setminus \overline{D}}|\nabla u_\e|^2+k^s|u_\e|^2 \, \text{d}x=0.
$$
This proves the claim, since $u_\e=u_\p = 0$ in $\mathbb{R}^2\setminus\overline{D}$ with $u^s=u_\e+u_\p$ in $\mathbb{R}^2\setminus\overline{D}$ and $u^s=0$  in $D$. 
\end{proof}

{We have been able to show that \eqref{direct2}--\eqref{direct4} can have at most one solution, but have not yet proven that a solution indeed exists. This will be addressed in the next section by means of the Fredholm alternative. The techniques used in the proof of Theorem \ref{up=0} are similar to other studies for biharmonic scattering. The main idea is to show that propagating and evanescent parts of the scattered field vanish outside the scatterer. This, together with the ``absorbing obstacle'' condition \eqref{absorbing} imply that the scattered field must also vanish on the interior of the scatterer. The fact that the material must be ``absorbing'' is vital for the above result. A further study of the direct scattering problem is needed to remove this assumption. }

\section{The biharmonic problem is Fredholm}\label{Fred}
In this section, we will address the well-posedness of our scattering problem \eqref{direct2}--\eqref{direct4}. To this end, notice that from the uniqueness result obtained in Theorem \ref{up=0} would imply well-posedness via the Fredholm alternative provided we can show that the scattering problem is Fredholm with index zero. So we wish to show that an equivalent variational formulation of \eqref{direct2}--\eqref{direct4} is associated with an operator that is the sum of a coercive and a compact operator. Toward this goal, an equivalent variational formulation that incorporates the radiation condition \eqref{direct4} will be derived.

Recall that for the scattered field $u^s \in H^2_{loc}(\mathbb{R}^2)$ satisfying \eqref{direct2}--\eqref{direct4} we have the variational formulation \eqref{varform1}, which we restate here for convenience:
{\small \begin{align*}
& a \big(u^s,\varphi \big)-k^4\!\!\int_{B_R}\!\!n u^s\overline{\varphi} \, \text{d}x - \int_{\partial B_R}\!\! \Big(\overline{\varphi}Nu^s+\partial_r \overline{\varphi} Mu^s\Big) \, \text{d}s = k^4\!\!\int_{D}\!\! (n-1)u^i \overline{\varphi}\, \text{d}x \quad \forall\,\varphi \in H^2(B_R)\, \\
& a(u,\varphi):=\int_{B_R}\Big( \mu \Delta u\Delta \overline{\varphi}+(1-\mu ) \sum_{i,j=1}^2{\partial_{x_{i} x_{j}} u} \, {\partial_{x_{i} x_{j}} \overline{\varphi}}\Big) \, \text{d}x\,,
\end{align*} }
where the surface differential operators $M$ and $N$ on $\partial B_R$ are defined in \eqref{defM} and \eqref{defN}.

In order to define an equivalent variational formulation of \eqref{direct2}--\eqref{direct4} we will consider an equivalent scattering problem defined in $B_R$. To do so, we need a way to enforce the radiation condition \eqref{direct4} in the truncated domain. Therefore, we define the  biharmonic Dirichlet-to-Neumann (DtN) operator 
\begin{alignat}{6}
\nonumber
\mathbb{T}: H^{3/2}(\partial B_R) \times H^{1/2}(\partial B_R) &\,\longrightarrow \,&&H^{-3/2}(\partial B_R)\times H^{-1/2}(\partial B_R)\\
\label{DtN}
\begin{pmatrix}
f \\
g
\end{pmatrix} & \, \longmapsto \, &&\begin{pmatrix}
N w|_{\partial B_R} \\
M w|_{\partial B_R}
\end{pmatrix}
\end{alignat}
where $w \in H^2_{loc}(\R^2 \setminus \overline{B}_R)$ is a radiating (i.e. satisfies \eqref{direct4}) solution to   
\begin{alignat}{6}
\label{directW1}
\Delta^2 w-k^4 w \,& =0 \qquad&&  \text{ in }\quad \mathbb{R}^2 \setminus \overline{B}_R \\
\label{directW2A}
w \big|_{\partial B_R} \,&= f \qquad&& \text{ on } \quad \partial B_R\\
\label{directW2B}
\partial_r w\big|_{\partial B_R} \,& = g \qquad&& \text{ on } \quad \partial B_R\,,
\end{alignat}
for any pair $f \in H^{3/2}(\partial B_R)$ and $g \in H^{1/2}(\partial B_R)$. The use of DtN maps to derive equivalent variational formulations for scattering problems is well known---see for instance \cite{approach,IST} for acoustic scattering. This particular DtN operator was studied in great detail in \cite{biwellposed} to prove the well-posedness for a clamped obstacle in the biharmonic context. 

Motivated by these works we see that the scattered field $u^s \in H^2(B_R)$ satisfies 
\begin{align}
\Delta^2 u^s-k^4 n(x)u^s=-k^4 \big(1-n(x)\big) u^i \quad & \text{in} \hspace{.2cm}  B_R  \label{directBR1}
\end{align}
along with the transmission conditions
\begin{align}
 [\![ u^s ]\!] =[\![\partial_\nu u^s ]\!] =[\![\Delta u^s ]\!]  =[\![\partial_\nu \Delta u^s ]\!]  =0 \label{directBR2}
\end{align}
as well as 
\begin{align}\label{directBR3}
\mathbb{T}\begin{pmatrix}
u^s|_{\partial B_R} \\
\partial_r u^s |_{\partial B_R}
\end{pmatrix}=\begin{pmatrix}
N u^s|_{\partial B_R} \\
M u^s|_{\partial B_R}
\end{pmatrix}.
\end{align}
The DtN mapping $\mathbb{T}$ enforces the radiation condition as $R \to \infty$. Notice that we have used the fact that the incident fields are $C^\infty$-smooth functions in the neighborhood of $\partial D$ to write the transmission conditions \eqref{directBR2} in terms of just the scattered field. We then have the following equivalence. 

\begin{theorem} \label{equiv-bvp}
If $u^s$ satisfies the direct scattering problem \eqref{direct2}--\eqref{direct4}, then its restriction to $B_R$ is a solution to \eqref{directBR1}--\eqref{directBR3}. Moreover, the solution to \eqref{directBR1}--\eqref{directBR3} can be extended to $\R^2$ in such a way that the extension satisfies \eqref{direct2}--\eqref{direct4}.
\end{theorem}
\begin{proof}
To prove the claim, we first let $u^s \in H^2_{loc}(\R^2)$ be a solution to \eqref{direct2}--\eqref{direct4}. It is clear that the restriction to $B_R$ satisfies \eqref{directBR1}--\eqref{directBR2}. Outside of $\overline{B_R}$, $u^s$ satisfies \eqref{directW1}--\eqref{directW2B} with boundary data $f \in H^{3/2}(\partial B_R)$ and $g \in H^{1/2}(\partial B_R)$ given by 
$$
f=u^s \big|_{\partial B_R}  \quad \text{ and } \quad g = \partial_r u^s \big|_{\partial B_R}
$$
along with the radiation condition \eqref{direct4}. This implies that $u^s$ satisfies \eqref{directBR1}--\eqref{directBR3}. Conversely, let $u^s \in H^2(B_R)$ satisfy \eqref{directBR1}--\eqref{directBR3}, then we can extend $u^s$ outside of $\overline{B_R}$ by solving 
$$
\Delta^2 u^s-k^4 u^s=0 \quad  \text{in}\quad \mathbb{R}^2 \setminus \overline{B}_R
$$
with boundary conditions 
$$
u^s_+ \big|_{\partial B_R} = u^s_- \big|_{\partial B_R}  \quad \text{ and } \quad \partial_r u^s_+  \big|_{\partial B_R} = \partial_r u^s_-  \big|_{\partial B_R}
$$
along with the radiation condition. Since the above boundary value problem is well-posed by \cite{biwellposed,novelbi}, we have extended $u^s$ into all of $H^2_{loc}(\R^2)$ such that it satisfies \eqref{direct2}--\eqref{direct4}, proving the claim. 
\end{proof} 

With the equivalence result in Theorem \ref{equiv-bvp}, the variational formulation \eqref{varform1} can be augmented to be equivalent to the direct scattering problem \eqref{direct2}--\eqref{direct4}. Therefore, the equivalent variational formulation is to find $u^s \in H^2(B_R)$ such that 
\begin{align}
a \big(u^s,\varphi \big)-k^4\int_{B_R}n u^s\overline{\varphi} \, \text{d}x - \int_{\partial B_R}\begin{pmatrix}
\varphi \\
\partial_{r}\varphi 
\end{pmatrix}^*\mathbb{T}\begin{pmatrix}
u^s \\
\partial_{r}u^s
\end{pmatrix}ds 
= k^4\int_{D} (n-1)u^i \overline{\varphi}\, \text{d}x \label{varform2}
\end{align}
for all $\varphi \in H^2(B_R)$. Here, the surface differential operators in the integral on $\partial B_R$ have been replaced with the DtN operator where the exponent ``$*$'' stands for the conjugate transpose. Again, this is equivalent to the direct scattering problem \eqref{direct2}--\eqref{direct4} by Theorem \ref{equiv-bvp} since the DtN operator is used to impose the radiation condition. In order to proceed, we need the following result for the DtN operator proven in Section 5 of \cite{biwellposed}. 
 
\begin{lemma}\label{DtN-lemma}
Let the DtN operator 
$$
\mathbb{T}: H^{3/2}(\partial B_R) \times H^{1/2}(\partial B_R) \longrightarrow H^{-3/2}(\partial B_R)\times H^{-1/2}(\partial B_R)
$$
be as defined in \eqref{DtN}. Then we have that 
$$
\int_{\partial B_R}\begin{pmatrix}
\varphi_2 \\
\partial_{r}\varphi_2 
\end{pmatrix}^*\mathbb{T}\begin{pmatrix}
\varphi_1 \\
\partial_{r}\varphi_1
\end{pmatrix}ds=t_0(\varphi_1,\varphi_2)+t_{\infty}(\varphi_1,\varphi_2) \quad \text{ for any } \,\, \varphi_1,\varphi_2 \in H^2(B_R)
$$
where the sesquilinear form $t_0 \big( \cdot \, , \, \cdot \big)$ is represented by a finite rank operator and the sesquilinear form $t_{\infty} \big( \cdot \, , \, \cdot \big)$ satisfies 
$$
-\mathrm{Re}\left\{ t_{\infty} (\varphi,\varphi) \right\} \geq 0  \quad \text{ for all } \,\, \varphi \in H^2(B_R). 
$$ 
\end{lemma}

Just as in the acoustic scattering case, the above result is obtained by appealing to separation of variables. This gives a way to derive an analytical formula for the DtN operator $\mathbb{T}$ which can then be analyzed explicitly. Indeed, it has been shown \cite{biwellposed} that there exists an $m_0 \in \N$ such that  
\begin{align*}
t_0(\varphi_1,\varphi_2)\,&=\frac{1}{R}\sum_{|m|\leq m_0}\begin{pmatrix}
\varphi_2^{(m)} \\
\partial_{r}\varphi_2^{(m)}
\end{pmatrix}^*\mathbb{T}_m \begin{pmatrix}
\varphi_1^{(m)} \\
\partial_{r}\varphi_1^{(m)}
\end{pmatrix}\\[1ex]
t_{\infty}(\varphi_1,\varphi_2) \,&=\frac{1}{R}\sum_{|m|> m_0}\begin{pmatrix}
\varphi_2^{(m)} \\
\partial_{r}\varphi_2^{(m)}
\end{pmatrix}^*\mathbb{T}_m \begin{pmatrix}
\varphi_1^{(m)} \\
\partial_{r}\varphi_1^{(m)}
\end{pmatrix}
\end{align*}
where $\varphi_j^{(m)}$ and $\partial_{r}\varphi_j^{(m)}$ are the Fourier coefficients for  $\varphi_j (R ,\theta)$ and $\partial_{r}\varphi_j (R ,\theta)$ for $j=1,2$, and the matrix $\mathbb{T}_m \in \C^{2\times 2}$ can be obtained explicitly and has a negative semidefinite real part for all $|m|> m_0$. 

In order to conclude that the equivalent variational formulation \eqref{varform2} is Fredholm with index zero we need a result pertaining to the sesquilinear form
$$
a(u,\varphi) :=\int_{B_R}\Big( \mu \Delta u\Delta \overline{\varphi}+(1-\mu ) \sum_{i,j=1}^2{\partial_{x_{i} x_{j}} u} \, {\partial_{x_{i} x_{j}} \overline{\varphi}}\Big) \, \text{d}x
$$
for a fixed constant $\mu$. Since it only includes the $L^2$ inner-product terms for the 2nd order derivatives, it is clear that $a \big( \cdot \, , \, \cdot \big)$ is not coercive on $H^2(B_R)$. However, as proven in \cite{LonEBVP}, for $\mu \in (-3,1)$ it does satisfy the classical G\r{a}rding inequality and is therefore weakly coercive (see \cite[Definition 6.65]{Salsa}).

\begin{lemma} \label{Garding}
For any $\mu\in (-3,1)$, there exists two constants $c_0>0$ and $\lambda_0 \geq 0 $ such that the sesquiliinear form $a\big( \cdot \, , \, \cdot \big)$ defined in \eqref{def-varA} satisfies the G\r{a}rding inequality 
$$
a(u,u)+\lambda_0\left\|u\right\|^2_{L^2(B_R)}\geq c_0\left\|u\right\|^2_{H^2(B_R)} \quad \text{for all }\quad u\in H^2(B_R).
$$
\end{lemma}

With these lemmas, we have all that we need to prove that the direct scattering problem \eqref{direct2}--\eqref{direct4} is Fredholm. This also implies well-posedness by appealing to the uniqueness result given in the previous section. We now state the main result of the paper. 

\begin{theorem} \label{wellposedness}
The direct scattering problem \eqref{direct2}--\eqref{direct4} is Fredholm of index zero and is therefore well-posed. Moreover, the scattered field satisfies the estimate 
$$
\| u^s \|_{H^2(B_R)} \leq C\| u^i \|_{L^2(D)}
$$
for some fixed constant $C>0$. 
\end{theorem}
\begin{proof}
To prove the claim, we consider the equivalent variational formulation \eqref{varform2}. From this we see that we can split the sesquilinear form such that 
$$
b(u^s,\varphi)=a(u^s,\varphi)+\lambda_0 \int_{B_R} u^s\overline{\varphi} \, \text{d}x-t_{\infty}(u^s,\varphi)
$$
and
$$
c(u^s,\varphi)=-\int_{B_R}(k^4n + \lambda_0)u^s\overline{\varphi} \, \text{d}x-t_0(u^s,\varphi).
$$
From the compact embedding of $H^2(B_R)$ into $L^2(B_R)$ \cite[Theorem 7.9]{Salsa} and the fact that $t_0 \big( \cdot \, , \, \cdot \big)$ is represented by a finite rank operator, it is clear that $c \big( \cdot \, , \, \cdot \big)$ is represented by a compact operator. Now, for $b \big( \cdot \, , \, \cdot \big)$ and for some fixed $\mu \in (-3,1)$, the G\r{a}rding inequality gives 
$$
\text{Re}\left\{b(u^s,u^s)\right\}=a(u^s,u^s)+\lambda_0\left \|u\right\|^2_{L^2(B_R)} -\text{Re}\left\{t_{\infty}(u^s,u^s)\right\}\geq c_0\left\|u^s\right\|^2_{H^2(B_R)}\,.
$$
Therefore, we have obtained that $b \big( \cdot \, , \, \cdot \big)$ is coercive on $H^2(B_R)$. This proves that the problem \eqref{direct2}--\eqref{direct4} is Fredholm of index zero. Moreover, by the conjugate linear functional given in the variational formulation \eqref{varform2} we see that 
$$
\left|k^4 \int_{D}(n-1)u^i \overline{\varphi} \, \text{d}x \right|  \leq  C \|u^i \|_{L^2(D)}\|\varphi \|_{H^2(B_R)}.
$$
This, combined with the Fredholm alternative and the uniqueness result, gives the well-posedness along with the norm estimate. 
\end{proof}

With this, we have proven that our scattering problem \eqref{direct2}--\eqref{direct4} is well-posed. This result uses techniques similar to the ones used for acoustic scattering problems but also have some new challenges. Indeed, even though the uniqueness result is handled in a somewhat similar manner as for acoustic scattering problems, we see that the ``absorbing'' assumption is necessary for our analysis but this is not the case for acoustic problems. It could be possible to use the limiting absorption principle (see e.g. \cite{kirsch-LAP} for analytical details) to remove the assumption on the imaginary part of $n$. This is indicative of the analytical challenges that come with studying biharmonic scattering problems. 

\section{Reciprocity Properties}\label{recip-rel}
Now that the well-posedness of the direct scattering problem has been established, we turn our attention to studying the properties of the solution. In this section, we derive two important reciprocity relations for the scattered field $u^s$. These are well known for acoustic scattering but have not been studied for the biharmonic scattering problem. These identities are used in inverse scattering to study the corresponding far-field (see \cite{HLL-biharmonic}) or near-field (see \cite{lsmKIRCH}) operators which are useful in the study of qualitative methods for recovering the scatterer $D$ from the measured scattering data. 

In order to study the reciprocity relations, we need to derive an integration by parts formula. By appealing to Green's second identity we obtain
\begin{align}\label{IBP}
\int_{\Omega} \paren{ {v} \Delta^2 u  -u\Delta^2{v}}  \, \text{d} x 
 = \int_{\partial \Omega} \paren{  \partial_\nu u \Delta {v} - u  \partial_\nu \Delta{v} 
 + {v}  \partial_\nu \Delta u  - \Delta u \partial_\nu {v} } \, \text{d} s
\end{align}
for any $u,v \in H^2(\Omega , \Delta^2)$ where $\partial \Omega$ is $C^2$-smooth and the Hilbert space was defined in \eqref{hilbertspace}. Notice that \eqref{IBP} is similar to the integration formula used in the previous section for $\mu=0$ and $\Omega=B_R$. 

With the aid of this formula, we will prove the reciprocity relationship for near-field data where we take
$$
u^i(x, {z})= \mathbb{G}(x,z;k) \quad \text{ for } \;y \in \Gamma
$$
and $\mathbb G(x,z;k)$ defined as in \eqref{fundamentalsol}. Recall that the scatterer $D$ is complactly enclosed by the measurement curve $\Gamma$, so that $\overline{D} \subset \text{int}(\Gamma)$ with $\Gamma$, and that the fundamental solution satisfies 
$$
\Delta^2 \mathbb{G}(\cdot,z;k) -k^4 \mathbb{G}(\cdot,z;k) = -\delta(\cdot - z) \quad \text{ in $\R^2$}
$$
along with the radiation condition \eqref{direct4}. With this we can now give our first reciprocity result for the scattered field given by a point source incident field. 

\begin{theorem}\label{firstreciprocity}
The scattered field $u^s$ solving \eqref{direct2}--\eqref{direct4} with incident field $u^i(\cdot , {z})= \mathbb{G}(\cdot ,z;k)$ satisfies the reciprocity relation $u^s(x,z)=u^s(z,x)$ for all $x,z\in\Gamma.$
\end{theorem}
\begin{proof}
To prove the claim, we let $u^s(\cdot , z)$ and  $u^s(\cdot , x)$ be the solutions to \eqref{direct2}--\eqref{direct4} with point sources located at $x,z \in \Gamma$. Then, using the integration by parts formula \eqref{IBP} for $v=\mathbb{G}(\cdot , x;k)$ and $u=u^s(\cdot , z)$ in the region $B_R\setminus\overline{D}$, letting $R\to\infty$ and using the radiation condition we obtain
\begin{align}
\nonumber
        u^s(x,z)=\,& -\int_{\partial D}\left(\Delta \mathbb{G}( y , x;k)\partial_{\nu}u^s(y,z) - {\partial_\nu }\Delta \mathbb{G}(y , x;k)u^s(y,z) \right) \text{ds}(y) \\
        \label{us1}
        \,& +\int_{\partial D}\left( \partial_{\nu} \mathbb{G}(y , x;k) \Delta u^s(y,z)-\mathbb{G}(y , x;k) {\partial_\nu }\Delta u^s(y,z)\right) \text{ds}(y)\,.
\end{align}
Reversing the roles of $x$ and $z$ in the argument above would yield
\begin{align}
\nonumber
        u^s(z,x)=\,& -\int_{\partial D}\left(\Delta \mathbb{G}( y , z;k)\partial_{\nu}u^s(y,x) - {\partial_\nu }\Delta \mathbb{G}(y , z;k)u^s(y,x) \right) \text{ds}(y) \\
        \label{us2}
        \,& +\int_{\partial D}\left( \partial_{\nu} \mathbb{G}(y , z;k) \Delta u^s(y,x)-\mathbb{G}(y , z;k) {\partial_\nu }\Delta u^s(y,x)\right) \text{ds}(y)\,.
\end{align}
Similarly, using integrating by parts formula \eqref{IBP} in the region $B_R\setminus\overline{D}$  with the functions $v = u^s(\cdot ,x)$ and $u=u^s(\cdot ,z)$ along with the radiation condition gives 
    \begin{align}
    \nonumber
       0= \,& -\int_{\partial D}\Delta u^s(y,x)\partial_{\nu}u^s(y,z)-{\partial_\nu }\Delta u^s(y,x)u^s(y,z)\text{ds}(y) \\
       \label{us3}
    \,&+\int_{\partial D}\partial_{\nu}u^s(y,x) \Delta u^s(y,z)-u^s(y,x){\partial_\nu }\Delta u^s(y,z)\text{ds}(y).
    \end{align}
Note that, by the transmission conditions \eqref{direct3}, the traces in the integral above are continuous across $\partial D$. Similarly, integrating by parts for $v=\mathbb{G}(\cdot , x;k)$ and $u=\mathbb{G}(\cdot,z;k)$ inside of the region $D$ gives 
    \begin{align}
    \nonumber
        0=&\,-\int_{\partial D}\Delta\mathbb{G}(y,x;k)\partial_{\nu}\mathbb{G}(y,z;k)-{\partial_\nu }\Delta \mathbb{G}(y,x;k)\mathbb{G}(y,z;k)\,\text{ds}(y) \\
        \label{us4}
    &\,+\int_{\partial D}\partial_{\nu}\mathbb{G}(y,x;k)\Delta\mathbb{G}(y,zk)-\mathbb{G}(y,x;k){\partial_\nu }\Delta \mathbb{G}(y,z;k)\,\text{ds}(y).
    \end{align}
Here we have used the fact that $x,z \notin D$ so that both fundamental solutions satisfy the same biharmonic equation $(\Delta^2-k^4)\mathbb{G}(y, \cdot;k) = 0$ in $D$. Now by the symmetry of the fundamental solution, adding \eqref{us1} and \eqref{us4} and subtracting \eqref{us3} from \eqref{us2} we obtain 
\begin{align}
\nonumber
u^s(x,z)-u^s(z,x)=&\, -\int_{\partial D}\partial_{\nu}u(y,z)\Delta u(y,x)-u(y,z){\partial_\nu }\Delta u(y,x) \,  \text{ds}(y)\\
\label{subtractionus}
&\,+\int_{\partial D}\Delta u(y,z)\partial_{\nu}u(y,x)-{\partial_\nu }\Delta u(y,z)u(y,x)\,\text{ds}(y)\,,
\end{align}
where $u(\cdot,\cdot)=\mathbb{G}(\cdot, \cdot;k)+u^s(\cdot,\cdot)$ is the total field. Since both total fields satisfy $(\Delta^2-k^4 n )u(y, \cdot) = 0$ in $D$, by appealing to the integration by parts formula \eqref{IBP} once again  we have that 
\begin{align*}
0=&\, -\int_{\partial D}\partial_{\nu}u(y,z)\Delta u(y,x)- u(y,z) {\partial_\nu }\Delta u(y,x)\,\text{ds}(y)\\
&\, +\int_{\partial D}\Delta u(y,z)\partial_{\nu}u(y,x)-{\partial_\nu }\Delta u(y,z)u(y,x)\,\text{ds}(y)\,.
\end{align*}
Since this is the right hand side of \eqref{subtractionus}, we have obtained $u^s(x,z)-u^s(z,x)=0$, proving the claim. 
\end{proof}

From, Theorem \ref{firstreciprocity} we see that the scattered field $u^s$ for a point source incident field inherited the symmetry on the measurement curve $\Gamma$ from the incident field.

Now we turn our attention to the case when the incident field $u^i(x, d)=\text{e}^{\text{i}kx \cdot d}$ with incident direction $d \in \mathbb{S}^1$. Recall that the  incident field satisfies 
$$\Delta^2 u^i(\cdot , d) -k^4u^i(\cdot , d) = 0 \quad \text{ in $\R^2$}$$
for any direction $d$ on the unit circle. With this in mind, we can prove another reciprocity relation for the far-field pattern $u^\infty$ of the scattered field. 

\begin{theorem}\label{secondreciprocity}
 Let $u^{\infty}(\hat{x},d)$ denote the far-field pattern defined by \eqref{far-fieldpattern} corresponding to the direct scattering problem \eqref{direct2}--\eqref{direct4}. Then
 $$
 u^{\infty}(\hat{x},d)=u^{\infty}(-d,-\hat{x}).
 $$ 
\end{theorem}
\begin{proof}
Here we let $u^s(\cdot ,d)$ and $u^s(\cdot ,-\hat{x})$ be solutions to \eqref{direct2}--\eqref{direct4} with incident directions $d$ and  $-\hat{x} \in \mathbb{S}^1$, respectively. We now recall the definition of the far-field pattern  
$$
u^{\infty}(\hat{x},d)=\int_{\partial B_R} u_\p(y,d)\partial_{\nu}u^i (y,-\hat{x})-\partial_{\nu}u_\p(y,d) u^i(y,-\hat{x}) \, \text{d}s(y)
$$
where $u_\p$ and $u_\e$ are the propagating and evanescent parts of the scattered field satisfying respectively the Helmholtz and modified Helmholtz equation which satisfies the Helmholtz equation in $\R^2\setminus \overline{D}$. We then arrive to the following 
\begin{align*}
-2k^2u^{\infty}(\hat{x},d)=&\, -2k^2\int_{\partial B_R} u_\p(y,d)\partial_{\nu_y}u^i(y,-\hat{x})-\partial_{\nu_y}u_\p(y,d) u^i(y,-\hat{x}) \, \text{d}s(y)\\
=&\,\phantom{-2k^2}\int_{\partial B_R} \Delta u_\p(y,d)\partial_{\nu}u^i(y,-\hat{x})-\partial_{\nu}\Delta u_\p(y,d) u^i(y,-\hat{x}) \, \text{d}s(y)\\
&\, \phantom{k^2}+\int_{\partial B_R} u_\p(y,d)\partial_{\nu}\Delta u^i(y,-\hat{x})-\partial_{\nu}u_\p(y,d) \Delta u^i(y,-\hat{x}) \, \text{d}s(y)\,
\end{align*}
where we have used that both $u^i(y,-\hat{x})$ and $u_\p (y,d)$ satisfy Helmholtz equation in $\R^2\setminus \overline{D}$. We then proceed by adding and subtracting the term  
$$
k^2\int_{\partial B_R} u_\e(y,d)\partial_{\nu}u^i(y,-\hat{x})-\partial_{\nu}u_\e(y,d) u^i(y,-\hat{x}) \, \text{d}s(y)
$$
to obtain
\begin{align}
\nonumber
-2k^2u^{\infty}(\hat{x},d)=&\,\phantom{+}\int_{\partial B_R} \Delta u^s(y,d)\partial_{\nu}u^i(y,-\hat{x})-\partial_{\nu}\Delta u^s(y,d) u^i(y,-\hat{x}) \, \text{d}s(y) \\
\label{farfield1}
&\, +\int_{\partial B_R} u^s(y,d)\partial_{\nu}\Delta u^i(y,-\hat{x})-\partial_{\nu}u^s(y,d) \Delta u^i(y,-\hat{x})\, \text{d}s(y)\,,
\end{align}
where we have used that $u_\e$ satisfies the modified Helmholtz equation in $\R^2\setminus \overline{D}$ as well as the fact that $u^i$ satisfies the Helmholtz equation in $\R^2$ and that $u^s=u_\p+u_\e$. Now, \eqref{farfield1} also implies that 
\begin{align}
\nonumber
-2k^2u^{\infty}(-d,-\hat{x})=&\,\phantom{+}\int_{\partial B_R} \Delta u^s(y,-\hat{x})\partial_{\nu}u^i(y,d)-\partial_{\nu}\Delta u^s(y,-\hat{x}) u^i(y,d)\, \text{d}s(y) \\
\label{farfield2}
&\, +\int_{\partial B_R} u^s(y,-\hat{x})\partial_{\nu}\Delta u^i(y,d)-\partial_{\nu}u^s(y,-\hat{x}) \Delta u^i(y,d)\, \text{d}s(y).
\end{align}
By applying the integration by parts formula \eqref{IBP} for $v=u^i(\cdot,d)$ and $uu^i(\cdot,-\hat{x})$ in $B_R$= we arrive at the identity 
\begin{align}
\nonumber
0=&\,-\int_{\partial B_R}u^i(y,d)\partial_{\nu}\Delta u^i(y,-\hat{x})-\Delta u^i(y,d)\partial_{\nu}u^i(y,-\hat{x})\text{d}s(y)\\
\label{ff3}
&\,+\int_{\partial B_R}u^i(y,-\hat{x})\partial_{\nu}u^i(y,d)-\Delta u^i(y,-\hat{x})\partial_{\nu}u^i(y,d)\text{d}s(y)
\end{align}
since they both solve the same biharmonic equation in $\R^2$. Similarly, the integration by parts formula \eqref{IBP} applied to $v=u^s(\cdot,d)$ and $u=u^s(\cdot,-\hat{x})$ in $B_R$ implies that 
\begin{align}
\nonumber
0=&\, -\int_{\partial B_R}u^s(y,d)\partial_{\nu}\Delta u^s(y,-\hat{x})-\Delta u^s(y,d)\partial_{\nu}u^s(y,-\hat{x})\, \text{d}s(y)\\
\label{ff4}
&\,+\int_{\partial B_R}u^s(y,-\hat{x})\partial_{\nu}\Delta u^s(y,d)-\Delta u^s(y,-\hat{x})\partial_{\nu}u^s(y,d) \, \text{d}s(y)
\end{align}
since again the two functions satisfy the same biharmonic equation in $\R^2$. Now, we can combine  \eqref{farfield1}--\eqref{ff4} to get the following 
\begin{align*}
-2k^2u^{\infty}(\hat{x},d)+ 2 k^2u^{\infty}(-d,-\hat{x})=&\, -\int_{\partial B_R}u(y,d)\partial_{\nu}\Delta u(y,-\hat{x})-\Delta u(y,d)\partial_{\nu}u(y,-\hat{x})\text{d}s(y)\\
&\,+\int_{\partial B_R}u(y,-\hat{x})\partial_{\nu}\Delta u^s(y,d)-\Delta u(y,-\hat{x})\partial_{\nu}u(y,d)\text{d}s(y)
\end{align*}
where $u=u^s+u^i$ denotes the total field. Just as in the previous result, the integration by part formula \eqref{IBP} along with the fact that $u(\cdot,d)$ and $u(\cdot,-\hat{x})$ satisfy the same biharmonic equation in $\R^2$ imply
\begin{align*}
0=&\,-\int_{\partial B_R}u(y,d)\partial_{\nu}\Delta u(y,-\hat{x})-\Delta u(y,d)\partial_{\nu}u(y,-\hat{x}) \, \text{d}s(y)\\
&\, +\int_{\partial B_R}u(y,-\hat{x})\partial_{\nu}\Delta u^s(y,d)-\Delta u(y,-\hat{x})\partial_{\nu}u(y,d) \, \text{d}s(y).
\end{align*}
Thus, we arrive to the desired result since we have shown 
$$-2k^2u^{\infty}(\hat{x},d)+ 2 k^2u^{\infty}(-d,-\hat{x})=0$$ 
for any $\hat{x},d \in \mathbb{S}^1$.  
\end{proof}

With these results, we see that the classical reciprocity relations for both point sources and plane waves hold for the biharmonic scattering problem \eqref{direct2}--\eqref{direct4}. Note that, even though for well-posedness the ``absorbing  assumption'' \eqref{absorbing} is required, the reciprocity relations do not depend on it. Therefore, Theorems \ref{firstreciprocity} and \ref{secondreciprocity} hold even for non-absorbing scatterers, provided that the direct problem is well-posed. \\

\noindent{\bf Analytical Example for Theorem \ref{secondreciprocity} on the Unit Disk:} Here we give an example to show the validity of the reciprocity relationship given in Theorems \ref{secondreciprocity}. To this end, we will assume that $D=B_1$ (i.e. the unit disk) where $n$ is given by a constant. This implies that the direct problem \eqref{direct2}--\eqref{direct3} for the scattered field $u^s$ outside the scatterer and the total field $u$ inside the scatterer, can be written as 
$$
\Delta^2 u^s-k^4 u^s= 0 \,\,\, \text{in $\R^2 \setminus \overline{B}_1$} \quad \text{ and } \quad  \Delta^2 u-k^4 n u= 0  \,\,\, \text{in ${B}_1$}.
$$
{ Denoting by $f(r,\theta)$ the value of the function $f$ at the point $x= r \left(\cos\theta,  \sin\theta \right)$,} the boundary conditions at $r=1$ are given by 
$$
u^s(1,\theta) - u(1,\theta)= -u^i (1,\theta ), \quad  \partial_r u^s(1,\theta) -  \partial_r u(1,\theta)= - \partial_r u^i (1,\theta)
$$
along with 
$$
\Delta u^s(1,\theta) - \Delta u(1,\theta)= -\Delta u^i (1,\theta) \quad \text{ and } \quad \quad  \partial_r \Delta u^s(1,\theta) -  \partial_r \Delta u(1,\theta)= - \partial_r \Delta u^i (1,\theta).
$$
Recall, the Jacobi--Anger expansion (see e.g. \cite{coltonkress}) for the incident plane wave 
$$u^{i} (r,\theta) = \sum\limits_{| \ell | =0}^{\infty} \text{i}^{\ell} J_{\ell}(kr)\text{e}^{\text{i}{\ell}(\theta - \phi)}
$$
{where the direction is given by $d = \left(\cos\phi, \sin\phi \right)$}. With this, we assume that the propagating and evanescent parts of the scattered field can be written as the Fourier series 
\begin{align*}
u_\p (r, \theta) = \sum\limits_{| \ell | =0}^{\infty}  \text{i}^{\ell}  a_{\ell} H_{\ell}^{(1)}(kr)\text{e}^{\text{i} \ell (\theta - \phi)} \quad \text{and} \quad u_\e  (r,\theta)  = \sum\limits_{| \ell | =0}^{\infty} \text{i}^{\ell}  b_{\ell} H_{\ell}^{(1)}({\mathrm{i}k}r)\text{e}^{\text{i} \ell (\theta - \phi)}.
\end{align*}
Similarly, we assume that the total field $u=u_{\text{H}} + u_{\text{M}}$ such that 
\begin{align*}
u_{\text{H}} (r, \theta) = \sum\limits_{| \ell | =0}^{\infty}  \text{i}^{\ell}  c_{\ell} J_{\ell}(k\sqrt[4]{n}r) \text{e}^{\text{i} \ell (\theta - \phi)} \quad \text{and} \quad u_{\text{M}}   (r,\theta)  = \sum\limits_{| \ell | =0}^{\infty} \text{i}^{\ell}  d_{\ell} J_{\ell}(\text{i} k\sqrt[4]{n}r)\text{e}^{\text{i} \ell (\theta - \phi)} 
\end{align*}
 where ${J}_\ell$ is the Bessel function of first kind. Notice that, we have $\Delta u_{\text{H}} =-k^2 \sqrt{n} u_{\text{H}}$ and $\Delta u_{\text{M}} =k^2 \sqrt{n} u_{\text{M}}$ in $B_1$.  Therefore, the four boundary conditions at $r=1$  impose the following system of equations 
\begin{equation*}\scalebox{0.95}{${\displaystyle
\begin{bmatrix}
H_{\ell}^{(1)}(k) & H_{\ell}^{(1)}({\mathrm{i}k})  & -J_{\ell}(k\sqrt[4]{n})  & - J_{\ell}(\text{i}k\sqrt[4]{n}) \\[1ex]
k{H_{\ell}^{(1)}}'(kr) & {\mathrm{i}k} {H_{\ell}^{(1)}}' ({\mathrm{i}k}) & -k\sqrt[4]{n}J'_{\ell}(k\sqrt[4]{n})  & - \text{i} k\sqrt[4]{n}J'_{\ell}(\text{i}k\sqrt[4]{n}) \\[1ex]
-k^2H_{\ell}^{(1)}(k) & k^2 H_{\ell}^{(1)}({\mathrm{i}k})  & k^2 \sqrt{n}J_{\ell}(k\sqrt[4]{n})  & - k^2 \sqrt{n}J_{\ell}(\text{i}k\sqrt[4]{n}) \\[1ex]
-k^3{H_{\ell}^{(1)}}'(kr) & {\mathrm{i}k^3} {H_{\ell}^{(1)}}' ({\mathrm{i}k}) & k^3(\sqrt[4]{n})^3 J'_{\ell}(k\sqrt[4]{n})  & - \text{i}  k^3(\sqrt[4]{n})^3J'_{\ell}(\text{i}k\sqrt[4]{n}) 
\end{bmatrix}\!
\begin{bmatrix}
a_{\ell}  \\[1.5ex] b_{\ell} \\[1.5ex] c_{\ell} \\[1.5ex] d_{\ell}
\end{bmatrix}\!
=  \!
\begin{bmatrix}  -J_{\ell}(k) \\[1.5ex] -kJ_{\ell}'(k) \\[1.5ex] k^2 J_{\ell}(k) \\[1.5ex] k^3 J_{\ell}'(k)  \end{bmatrix} }$ }
\end{equation*}
{for all } ${\ell} \in \Z$. This implies that, by solving the system for the Fourier coefficients (in particular $a_\ell$) the far-field pattern is given by 
$$u^{\infty} (\theta,\phi) = \frac{ 4 }{ \text{i} }\sum\limits_{| {\ell} | =0}^{\infty}  a_{\ell}   \text{e}^{\text{i}{\ell}(\theta - \phi)}$$
from the asymptotics of the Hankel functions. This implies that, the $u^{\infty} (\theta,\phi)  = u^{\infty} (\phi+\pi,\theta+\pi)$ by appealing to Theorem \ref{secondreciprocity}.
Here, to approximate $u^{\infty}$ we take the truncated series for $|\ell| =0, \cdots , 10$. 

For our examples, we use \texttt{Matlab} to compute the Fourier coefficients using the ``$\setminus$'' command. We can now see from Table \ref{circleTable1} the above reciprocity relationship.
\begin{table}[!ht]
\centering
 \begin{tabular}{r|c|c}
 $(\theta,\phi )$  & $u^{\infty} (\theta,\phi)$  & $u^{\infty} (\phi+\pi,\theta+\pi)$\\
\hline
$(\pi/2,\pi)$ & $0.0414 + 0.6795\text{i} $ &  $0.0414 + 0.6795\text{i} $\\
$(\pi/3,\pi/2)$ & $5.3133 + 7.3659\text{i}$ & $5.3133 + 7.3659\text{i}$\\
$(\sqrt{\pi},\pi/2)$ & $0.4839 + 1.8926\text{i}$ & $0.4839 + 1.8926\text{i}$\\
$(1,1/2)$ & $5.4547 + 7.4765\text{i}$ & $5.4547 + 7.4765\text{i}$\\
\hline
 \end{tabular}
 \caption{Numerical validation of Theorem \ref{secondreciprocity} where $k=2$ and $n=4+\text{i}$.} \label{circleTable1}
\end{table}
From Table \ref{circleTable1} we see that numerically the relationship valid for at least four digits. 

Recall, that Theorem \ref{secondreciprocity} is valid whenever the direct problem is well-posed (i.e. possibly for non-absorbing scatterers). We provide Table \ref{circleTable2} as further validation of the reciprocity relationship. Here, consider the far-field matrices
$$\textbf{F}_1=\Big[u^\infty(\theta_i , \phi_j)\Big]^{64}_{i,j=1} \quad \text{and} \quad \textbf{F}_2=\Big[u^\infty(\phi_j +\pi , \theta_i + \pi)\Big]^{64}_{i,j=1}$$ 
with $\theta_i=2\pi(i-1)/64$, and $\phi_j=2\pi(j-1)/64$ for $i,j=1,\dots,64$. Therefore, we should have that $\|\textbf{F}_1 - \textbf{F}_2 \|_2 \approx 0$. Now, in Table \ref{circleTable2} for multiple pairs of wave numbers $k$ and refractive indices $n$ we see that the far-field matrices are approximately the same, even for real valued $n$. 
\begin{table}[!ht]
\centering
 \begin{tabular}{r|c}
 $(k,n)$  & $\|\textbf{F}_1 - \textbf{F}_2 \|_2$ \\
\hline
$(2,4+\text{i})$ &  $9.9958 \times 10^{-14}$\\
$(2,4)$ & $1.1834\times 10^{-13}$ \\
$(\pi,16+\text{i}/2)$  & $3.3697\times 10^{-13}$ \\
$(\pi,16)$ & $3.5576\times 10^{-13}$ \\
$(4, 2+2\text{i})$  & $2.5771\times 10^{-13}$ \\
$(4,2)$ & $2.7684\times 10^{-13}$ \\
\hline
 \end{tabular}
 \caption{Norm of the difference in the  far-field matrices for various $k$ and $n$.} \label{circleTable2}
\end{table}
\\ 

\noindent{\bf Analytical Example for Theorem \ref{firstreciprocity} on the Unit Disk:} Now, for this case we can work similarly to the previous example. Therefore, we recall that for $|z| > |x|$ we have the expansion 
$$H^{(1)}_0( \tau | x-z|)=\sum\limits_{|\ell|=0}^{\infty}  H^{(1)}_\ell( \tau |z|) {J}_\ell\left(\tau |x|  \right) \text{e}^{\text{i}\ell(\theta-\phi)}$$
for a given constant $\tau$ (see e.g. \cite{coltonkress}) where again ${J}_\ell$ is the Bessel function of first kind and $H^{(1)}_\ell$ is the first kind Hankel function of order $\ell$. Here, we let $\Gamma = \partial B_2$ and be we will again assume that $D=B_1$. This implies that, in polar coordinates $x= r \big(\cos(\theta),  \sin(\theta) \big)$ and $z = 2\big(\cos(\phi), \sin(\phi) \big)$. 

Taking an incident wave generated by a fundamental solution, and substituting the series expansion above into the definition \eqref{fundamentalsol}, it follows that the point source located at $z  \in \partial B_2$ is given by
$$
u^i(r ,\theta ) = \mathbb G(x,z) =\frac{\text{i}}{8k^2}\sum\limits_{|\ell|=0}^{\infty}   \left[ H^{(1)}_\ell(2 k) {J}_\ell\left(k r  \right) - H^{(1)}_\ell( 2 \text{i} k ) {J}_\ell\left(\text{i} k r  \right)\right] \text{e}^{\text{i}\ell(\theta-\phi)}.
$$
As before, we assume that the propagating and evanescent parts of the scattered field can be written as the Fourier series 
\begin{align*}
u_\p (r, \theta) = \sum\limits_{| \ell | =0}^{\infty} a_{\ell}  H_{\ell}^{(1)}(kr)\text{e}^{\text{i} \ell (\theta - \phi)} \quad \text{and} \quad u_\e  (r,\theta)  = \sum\limits_{| \ell | =0}^{\infty}   b_{\ell} H_{\ell}^{(1)}({\mathrm{i}k}r)\text{e}^{ \text{i} \ell (\theta - \phi)}.
\end{align*}
Similarly, we assume that the total field $u=u_{\text{H}} + u_{\text{M}}$ such that 
\begin{align*}
u_{\text{H}} (r, \theta) = \sum\limits_{| \ell | =0}^{\infty}   c_{\ell} J_{\ell}(k\sqrt[4]{n}r) \text{e}^{ \text{i} \ell (\theta - \phi)} \quad \text{and} \quad u_{\text{M}}   (r,\theta)  = \sum\limits_{| \ell | =0}^{\infty}  d_{\ell} J_{\ell}( \text{i} k\sqrt[4]{n}r)\text{e}^{\text{i} \ell (\theta - \phi)}.
\end{align*}
By applying the boundary conditions at $r=1$ gives that to find the Fourier coefficients, we solve the system of from the previous example where the new right hand side is given by 
$$
- \frac{\text{i}}{8k^2}  \begin{bmatrix} H^{(1)}_\ell(2 k) {J}_\ell\left(k   \right) - H^{(1)}_\ell( 2 \text{i} k ) {J}_\ell\left(\text{i} k  \right) \\ 
k H^{(1)}_\ell(2 k) {J}'_\ell\left(k   \right) - \text{i} k H^{(1)}_\ell( 2 \text{i} k ) {J}'_\ell\left(\text{i} k   \right) \\ 
-k^2 H^{(1)}_\ell(2 k) {J}_\ell\left(k   \right) +k^2  H^{(1)}_\ell( 2 \text{i} k ) {J}_\ell\left(\text{i} k  \right) \\ 
- k^3 H^{(1)}_\ell(2 k) {J}'_\ell\left(k   \right) +\text{i} k^3 H^{(1)}_\ell( 2 \text{i} k ) {J}'_\ell\left(\text{i} k \right) \end{bmatrix} .
$$
This is obtained as before, where we again use the fact the the solution and the incident field are given by the sum of a solution to the Helmholtz equation and the modified Helmholtz equation. 
 
Now to provide numerical evidence, we again use \texttt{Matlab} to compute the Fourier coefficients using the ``$\setminus$'' command for the 4$\times$4 system of equations for each $\ell \in \Z$. With this in mind, notice that the near-field data 
$$u^s (x,z)  \quad \text{with} \quad \text{ $x= 2 \big(\cos(\theta),  \sin(\theta) \big)$ and $z = 2\big(\cos(\phi), \sin(\phi) \big)$. }$$ 
is given by $u^s (\theta,\phi)$.  From the Fourier series expansions we have that 
$$
u^s (\theta,\phi) =  \sum\limits_{| \ell | =0}^{\infty} \left[ a_{\ell}  H_{\ell}^{(1)}(2k) +    b_{\ell} H_{\ell}^{(1)}(2{\mathrm{i}k}) \right]\text{e}^{\text{i} \ell (\theta - \phi)}.
$$
We will again, take the truncated series for $|\ell| =0, \cdots , 10$ in our examples. With this, it is clear that by Theorem \ref{firstreciprocity} that $u^s (\theta,\phi) = u^s (\phi, \theta)$. This is shown in Table \ref{circleTable3} where we check the equality. Again we see that numerically, the relationship valid for at least four digits.  

\begin{table}[!ht]
\centering
 \begin{tabular}{r|c|c}
$(\theta,\phi )$   & $u^{\infty} (\theta,\phi)$  & $u^{s} (\phi,\theta)$\\
\hline
$(\pi/2,\pi)$ & $0.0018 + 0.0227\text{i} $ &  $0.0018 + 0.0227\text{i} $\\
$(\pi/3,\pi/2)$ & $-0.0026 - 0.0053\text{i}$ & $-0.0026 - 0.0053\text{i}$\\
$(\sqrt{\pi},\pi/2)$ & $0.0083 - 0.0001\text{i}$ &$0.0083 - 0.0001\text{i}$\\
$(1,1/2)$ & $-0.0017 - 0.0051\text{i}$ & $-0.0017 - 0.0051\text{i}$\\
\hline
 \end{tabular}
 \caption{Numerical validation of Theorem \ref{firstreciprocity} where $k=2$ and $n=4+\text{i}$.} \label{circleTable3}
\end{table}

Now, we show the reciprocity relationship for various values of $k$ and $n$, both constant. To this end, we will define the near-field matrices in order to check the relationship. We now define the matrices  
$$\textbf{N}_1=\Big[u^s(\theta_i , \phi_j)\Big]^{64}_{i,j=1} \quad \text{and} \quad \textbf{N}_2=\Big[u^s(\phi_j , \theta_i )\Big]^{64}_{i,j=1}$$ 
i.e. $\textbf{N}_1 = \textbf{N}_2^\top$, where we take $\theta_i=2\pi(i-1)/64$, and $\phi_j=2\pi(j-1)/64$ for $i,j=1,\dots,64$. Here, the reciprocity relationship $u^s (\theta,\phi) = u^s (\phi, \theta)$ implies that $\|\textbf{N}_1 - \textbf{N}_2 \|_2 \approx 0$ for any pair of values $(k,n)$, provided that the scattering problem is well-posed. Just as in the previous example, we will provide numerics for both an absorbing and non-absorbing scatterer. In Table \ref{circleTable4}, we see that the near-field matrices are approximately the same for various choices in parameters. 
\begin{table}[H]
\centering
 \begin{tabular}{r|c}
$(k,n )$ & $\|\textbf{N}_1 - \textbf{N}_2 \|_2$ \\
\hline
$(2,4+\text{i})$ &  $6.7592 \times 10^{-17}$\\
$(2,4)$ & $7.8676\times 10^{-17}$ \\
$(\pi,16+\text{i}/2)$  & $5.4290\times 10^{-17}$ \\
$(\pi,16)$ & $6.7592\times 10^{-17}$ \\
$(4, 2+2\text{i})$  & $6.6802\times 10^{-17}$ \\
$(4,2)$ & $7.4389\times 10^{-17}$ \\
\hline
 \end{tabular}
 \caption{Norm of the difference in the near-field matrices for various $k$ and $n$.} \label{circleTable4}
\end{table}

With this, we have shown analytically and numerically that the two reciprocity relationship are valid. Again, the reciprocity relationship are valid as long as the forward problem is well-posed. We have tested our theorems with both absorbing and non-absorbing parameters. \\


\noindent{\bf Numerical Example for Theorem \ref{firstreciprocity} in the case of a non-circular scatterer:} If the material properties in $D$ and $ \R^2 \setminus\overline{D}$ are such that the wave numbers on each region are the constants $\tau_- :=   k\sqrt[4]{n}$ in $D$ and $\tau_+ := k$ in $ \R^2 \setminus\overline{D}$, the total out of plane displacement $u$ can be split in the form
$$
u = \begin{cases} u^s + u^i & \text{ in } \R^2 \setminus\overline{D},\\ u^t & \text{ in } D \,,\end{cases} 
$$
where the superscripts $s$ and $t$ stand respectively for ``scattered'' and ``transmitted''. Moreover, the total wave $u$ satisfies the biharmonic wave equation with coefficient
$$
\tau = \begin{cases} \tau_+ & \text{ in }  \R^2 \setminus\overline{D} \\ \tau_- & \text{ in } D \end{cases}
$$
over the entire plane and it must satisfy the continuity conditions \eqref{directBR2} across the interface $\partial D$. Therefore, it follows that the unknown scattered and transmitted fields satisfy
\begin{alignat*}{6}
-\Delta^2 u^s + \tau_+^4 u^s &= 0 \qquad&& \text{ in }  \R^2 \setminus\overline{D}\,,\\
-\Delta^2 u^t + \tau_-^4 u^t &= 0 \qquad&& \text{ in } D\,,\\
u^t_-\big|_{\partial D} - u^s_+\big|_{\partial D} & = u^i_+\big|_{\partial D} \qquad&& \text{ on } {\partial D}\,,\\
\partial_\nu^-u^t - \partial_\nu^+u^s &= \partial_\nu^+u^i \qquad&& \text{ on } {\partial D}\,,\\
\Delta u^t_-\big|_{\partial D} - \Delta u^s_+\big|_{\partial D} &= \Delta u^i_+\big|_{\partial D} \qquad&& \text{ on } \partial D\,,\\
\partial_\nu^-\Delta u^t - \partial_\nu^+\Delta u^s &= \partial_\nu^+\Delta u^i \qquad&& \text{ on } \partial D\,,
\end{alignat*}
along with the radiation conditions
$$
\lim_{|x| \to \infty} \sqrt{r}(\partial_r u^s - \text{i} \tau_+ u^s) = 0 \qquad \text{ and } \qquad \lim_{r \to \infty} \sqrt{r}(\partial_r\Delta u^s - \text{i} \tau_+ \Delta u^s) = 0\,.
$$
By an argument analogous to that of Section \ref{direct-prob-LDSM}, the scattered and transmitted fields can be split into propagating and evanescent components as
$$
u^{s} =  u^{s}_{\text{pr}} + u^{s}_{\text{ev}} \qquad \text{ and } \qquad u^{t} =  u^{t}_{\text{pr}} + u^{t}_{\text{ev}}.
$$
The propagating and evanescent components satisfy the system
\begin{subequations}\label{eq:system}
\begin{alignat}{6}
\label{eq:systemA}
-\Delta u^s_{\text{pr}} - \tau_+^2 u^s_{\text{pr}} &= 0 \qquad&& \text{ in } \R^2 \setminus\overline{D}\,,\\[1ex]
\label{eq:systemB}
-\Delta u^s_{\text{ev}} + \tau_+^2 u^s_{\text{ev}} &= 0 \qquad&& \text{ in }  \R^2 \setminus\overline{D}\,,\\[1ex]
\label{eq:systemC}
-\Delta u^t_{\text{pr}} - \tau_-^2 u^t_{\text{pr}} &= 0 \qquad&& \text{ in } D\,,\\[1ex]
\label{eq:systemD}
-\Delta u^t_{\text{ev}} + \tau_-^2 u^t_{\text{ev}} &= 0 \qquad&& \text{ in } D\,,\\[1ex]
\label{eq:systemE}
\left(u^t_{\text{pr}} + u^t_{\text{ev}}\right)_-\big|_{\partial D} - \left(u^s_{\text{pr}} + u^s_{\text{ev}}\right)_+\big|_{\partial D} &= u^i_+\big|_{\partial D} \qquad&& \text{ on } \partial D\,,\\[1ex]
\label{eq:systemF}
\partial_\nu^-\left(u^t_{\text{pr}} + u^t_{\text{ev}}\right) - \partial_\nu^+\left(u^s_{\text{pr}} + u^s_{\text{ev}}\right) &= \partial_\nu^+u^i \qquad&& \text{ on } \partial D\,,\\[1ex]
\label{eq:systemG}
\tau_-^2\left(-u^t_{\text{pr}} + u^t_{\text{ev}}\right)_-\big|_{\partial D} - \tau_+^2\left(-u_{\text{pr}} + u_{\text{ev}}\right)_+\big|_{\partial D} &= \Delta u^i_+\big|_{\partial D} \qquad&& \text{ on } \partial D\,,\\[1ex]
\label{eq:systemH}
\tau_-^2\partial_\nu^- \left(- u^t_{\text{pr}} + u^t_{\text{ev}}\right) - \tau_+^2\partial_\nu^+ \left(-u^s_{\text{pr}} + u^s_{\text{ev}}\right) &= \partial_\nu^+\Delta u^i \qquad&& \text{ on } \partial D\,,
\end{alignat}
together with the radiation conditions
\begin{equation}\label{eq:systemI}
\lim_{r \to \infty} \sqrt{r}\left(\partial_r u^s_{\text{pr}} - \text{i} \tau_+ u^s_{\text{pr}}\right) = 0 \qquad \text{ and } \qquad \lim_{r \to \infty} \sqrt{r}\left(\partial_ru^s_{\text{ev}} - \text{i} \tau_+ u^s_{\text{ev}}\right) = 0.
\end{equation}
\end{subequations}
Note that equations \eqref{eq:systemA} and \eqref{eq:systemB} were used to rewrite the continuity conditions involving the Laplacian and its normal derivative in terms of $u^t$ and $u^s$ and the corresponding wave numbers, resulting in \eqref{eq:systemG} and \eqref{eq:systemH}.

The absence of a source term in the PDEs \eqref{eq:systemA} through \eqref{eq:systemD} suggests a boundary integral treatment. Following \cite{novelbi}, we propose integral representations of the form
\begin{alignat*}{8}
u^s_{\text{pr}} &= \int_{\partial D} \partial_{\nu(y)}\mathbb G_{\text{pr}}(x,y; \tau_+)\varphi_s (y) \,\text{d}s(y)\,, \qquad\qquad&& 
u^s_{\text{ev}} &= -\int_{\partial D} \mathbb G_{\text{ev}}(x,y; \tau_+)\lambda_s (y) \,\text{d}s(y)\, \\
u^t_{\text{pr}} &= -\int_{\partial D} \partial_{\nu(y)}\mathbb G_{\text{pr}}(x,y; \tau_-)\varphi_t  (y) \,\text{d}s(y)\,, \qquad\qquad&& 
u^t_{\text{ev}} &= \int_{\partial D} \mathbb G_{\text{ev}}(x,y ; \tau_-)\lambda_t  (y) \, \text{d}s(y)\,.
\end{alignat*}
In the representations above,
$$
\mathbb G_{\text{ev}}(x,y;k) := \frac{\text{i}}{4} H^{(1)}_0(k|x-y|)\quad \text{ and } \quad \mathbb G_{\text{pr}}(x,y;k) := \frac{\text{i}}{4} H^{(1)}_0(\text{i}k|x-y|)
$$
are the fundamental solutions of the Helmholtz and modified Helmholtz equation, and the densities 
$$
\varphi_s,\varphi_t\in H^{1/2}({\partial D}) \qquad \text{ and } \qquad \lambda_s,\lambda_t\in H^{-1/2}({\partial D})
$$
 are unknowns to be determined using the transmission conditions \eqref{eq:systemE}--\eqref{eq:systemH}.
 
We will now use standard arguments from potential theory that we shall not detail here (see, for instance \cite{HsWe2021}) to express the transmission conditions in terms of the unknown densities. Doing so leads to the following system of boundary integral equations 

{\small \begin{equation}\label{eq:BIsystem}\scalebox{0.92}{$
{\displaystyle 
\hspace{-0.95cm}\begin{bmatrix}
 \mathcal V(\tau_-) &-\left( \frac{1}{2}\mathcal I - \mathcal K(\tau_-)\right)  &\mathcal V(\tau_+)  & -\left( \frac{1}{2}\mathcal I + \mathcal K(\tau_+) \right)\\[1.5ex]
\left(\frac{1}{2} + \mathcal J(\tau_-) \right) & \mathcal W(\tau_-)  & -\left( \frac{1}{2}\mathcal I + \mathcal J(\tau_+)\right) & \mathcal W(\tau_+) \\[1.5ex]
\tau_-^2\mathcal V(\tau_-) & \tau_-^2\left( -\frac{1}{2}\mathcal I + \mathcal K(\tau_-)\right) & \tau_+^2\mathcal V(\tau_+) & \tau_+^2\left( \frac{1}{2}\mathcal I + \mathcal K(\tau_+)\right) \\[1.5ex]
\tau_-^2\left( \frac{1}{2}\mathcal I + \mathcal J(\tau_-)\right)  & -\tau_-^2\mathcal W(\tau_-)  & \tau_+^2\left(-\frac{1}{2}\mathcal I + \mathcal J(\tau_+)\right)  & -\tau_+^2\mathcal W(\tau_+)
\end{bmatrix}\!\!
\begin{bmatrix}
\lambda_t\\[1.5ex] \varphi_t \\[1.5ex] \lambda_s\\[1.5ex] \varphi_s  
\end{bmatrix}\!
=\!
\begin{bmatrix}
u^i \\[1.5ex] \partial_\nu^+ u^i \\[1.5ex] \Delta u^i \\[1.5ex] \partial_\nu^+ \Delta u^i
\end{bmatrix}\,, }$}
\end{equation} }
\noindent where $\mathcal I$ is the identity operator and the boundary integral operators appearing in the matrix are defined for $\xi\in\mathbb C$, $\lambda\in\,H^{-1/2}({\partial D})$, and  $\varphi\in\,H^{1/2}({\partial D})$ by:
\begin{alignat*}{10}
\mathcal V(\xi)\lambda :=\,& \int_{\partial D} \mathbb G_{\text{ev}}(x,y;\xi)\lambda (y) \, \text{d}s(y)\,, \qquad&& 
\mathcal J(\xi)\lambda :=\,& \int_{\partial D} \partial_{\nu(x)}\mathbb G_{\text{ev}}(x,y;\xi)\lambda(y)\,\text{d}s(y)\,,\\[1ex]
\mathcal W(\xi)\varphi :=\,& \int_{\partial D} \partial_{\nu(x)}\partial_{\nu(y)}\mathbb G_{\text{pr}}(x,y;\xi)\varphi (y)\,\text{d}s(y)\,, \qquad&& 
\mathcal K(\xi)\varphi :=\,& \int_{\partial D} \partial_{\nu(y)}\mathbb G_{\text{pr}}(x,y;\xi)\varphi (y)\,\text{d}s(y)\,.
\end{alignat*}
The detailed analysis of the system \eqref{eq:BIsystem} is beyond the scope of this article and will be the subject of a separate communication \cite{SV2025}. The system was solved numerically using {\tt deltaBEM} \cite{DoLuSa2014a,deltaBEM,DoLuSa2014b,DoSaSa2015}, a third-order collocation-based suite for the discretization of boundary integral equations. The discretization employed 500 points for the scattering geometry. 

For this example we will consider that $D$ is a kite-shaped inclusion with boundary $\partial D$ given by the parametrization
$$
\partial D : = \left(\cos\left(2\pi t\right) + \cos\left(4\pi t\right), 2\sin\left(2\pi t\right)\right) \qquad \text{ for }\; t\in[0,1)
 $$
 and the measurement curve $\Gamma$ is the circle centered at $0$ with radius $r=3$---as depicted in the left panel of Figure \ref{fig:NumericalExample}. Various values of the coefficients $\tau_-$ and $\tau_+$ are considered for fundamental incident waves of the form \eqref{fundamentalsol} with source points located at
$$ 
z_i = 3(\cos\theta_i,\sin\theta_i)\in\Gamma \quad \text{ for } \quad \theta_i = 2\pi(i-1)/64. 
$$ 
The central and right panels of Figure \ref{fig:NumericalExample} depict the real part and the magnitude of the real part of the total field generated by one such wave with source at $z=(-3,0)$ and interior/exterior coefficients $\tau_-=5$ and $\tau_+=15$. Table \ref{kiteTable} provides verification to the reciprocity relation for the matrices $\mathbf N_1$ and $\mathbf N_2$ defined in the previous example for various values of  $\tau_-$ and $\tau_+$. \\

\begin{figure}
\centering
\includegraphics[width = 0.28\linewidth]{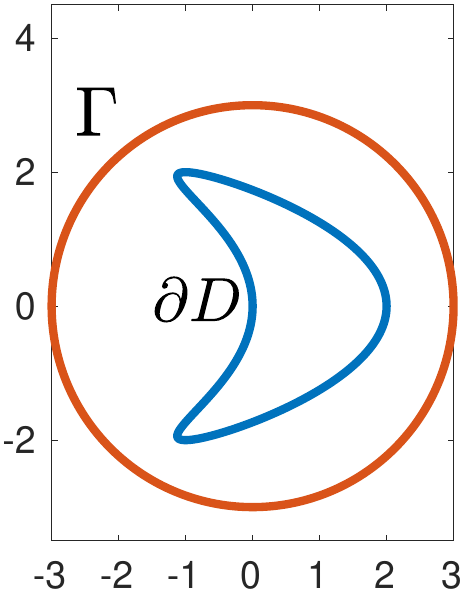}  \qquad\quad
\includegraphics[width = 0.27\linewidth]{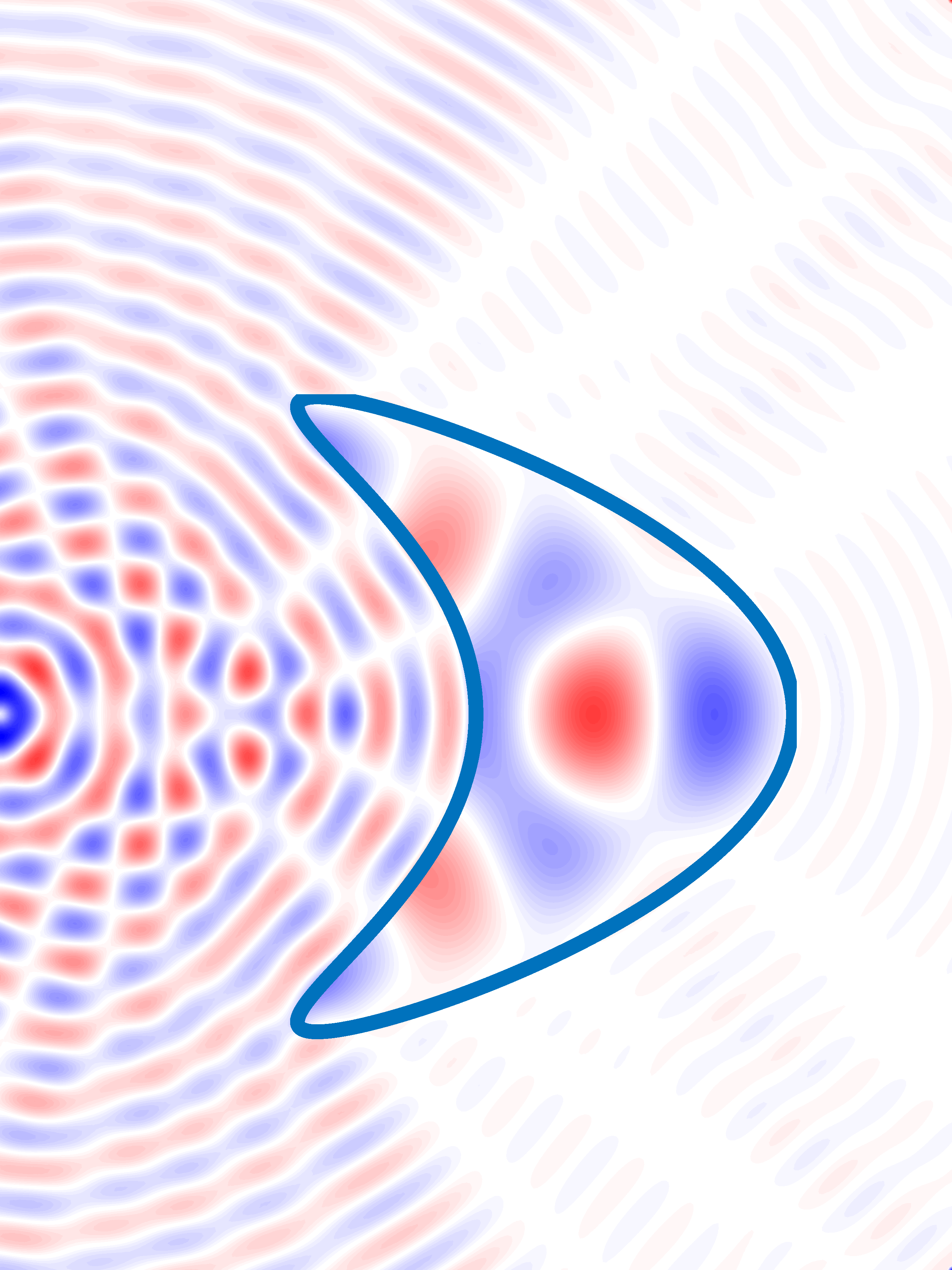}\qquad\quad
\includegraphics[width = 0.27\linewidth]{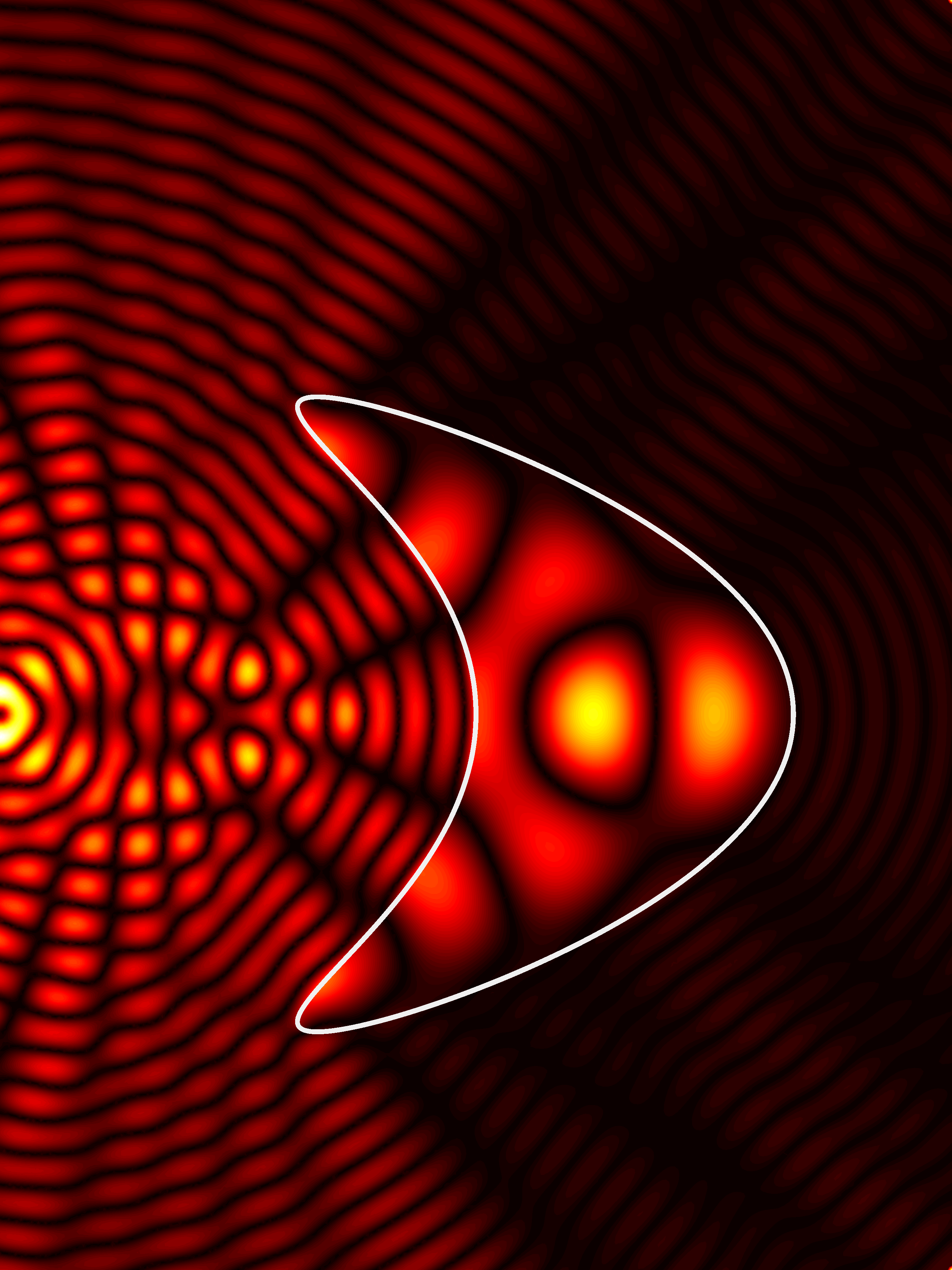} 
\caption{Left: Boundary of the kite-shaped scatterer, $\partial D$, and measurement curve, $\Gamma$. Center and Right: An incident fundamental wave located at $z=(-3,0)$ propagates through a medium with $\tau_+= 15$ and impinges upon a scatterer with $\tau_-=5$. The center panel shows the real part of the total wavefield (blue is negative, red is positive), while the right panel shows the magnitude of the real part.}\label{fig:NumericalExample}
\end{figure}
 
 \begin{table}[ht]
\centering
 \begin{tabular}{r|c}
  $(\tau_-,\tau_+ )$ & $\|\textbf{N}_1 - \textbf{N}_2 \|_2$ \\
\hline
$(5,15)$ &  $1.12 \times 10^{-4}$\\
$(5+\pi,15-\pi)$ & $1.53\times 10^{-4}$ \\
$(5+2\pi,15-2\pi)$  & $2.56\times 10^{-4}$ \\
$(9.5,10.5)$ & $7.00\times 10^{-5}$ \\
$(10.5, 9.5)$  & $8.91\times 10^{-5}$ \\
$(11,17.7)$ & $6.16\times 10^{-4}$ \\
\hline
 \end{tabular}
 \caption{Norm of the difference in the numerical approximation of the near-field matrices for various $\tau_-$ and $\tau_+$.}\label{kiteTable}
\end{table}

\noindent{\bf Acknowledgments:} 

\noindent{The authors are thankful for the support of the U. S. National Science Foundation. The research of Isaac Harris is partially supported by the NSF DMS Grant 2208256. Tonatiuh S\'anchez-Vizuet was partially funded through the grant NSF-DMS-2137305. 


\end{document}